\documentclass[hidelinks,12pt,a4paper,reqno]{amsart}
\usepackage{graphicx}
\usepackage[all]{xy}
\usepackage{amssymb}
\usepackage{amsmath}
\usepackage[utf8]{inputenc}
\usepackage{amsthm}

\newtheorem{proposition}{Proposition}
\newtheorem{corollary}{Corollary}
\newtheorem{remark}{Remark}
\newtheorem{definition}{Definition}
\newtheorem{theorem}{Theorem}

\usepackage[backref=page]{hyperref}
\renewcommand*{\backref}[1]{}
\renewcommand*{\backrefalt}[4]{%
    \ifcase #1 (Not cited.)%
    \or        (Cited on page~#2.)%
    \else      (Cited on pages~#2.)%
    \fi}
\newcommand{\C}{\mathbf{C}}

\newcommand{\N}{\mathbf{N}}

\newcommand{\A}{\mathcal{R}}
\newcommand{\V}{\mathcal{V}}

\newcommand{\Mon}{\mathbf{Mon}}
\newcommand{\Ord}{\mathbf{Ord}}
\newcommand{\RGraph}{\mathbf{RGraph}}

\newcommand{\Mono}{\mathbf{Mono}}
\newcommand{\NMono}{\mathbf{NMono}}

\newcommand{\RRel}{\mathbf{RRel}}
\newcommand{\Eq}{\mathbf{Eq}}
\newcommand{\Rel}{\mathbf{Rel}}
\newcommand{\Fix}{\mathbf{Fix}}

\newcommand{\Set}{\mathbf{Set}}
\renewcommand{\S}{\mathbf{S}}

\begin{document}

\title[On the normality of monoid monomorphisms]{On the normality of monoid monomorphisms}


\author[N. Martins-Ferreira]{Nelson Martins-Ferreira}
\address[Nelson Martins-Ferreira]{Instituto Politécnico de Leiria, Leiria, Portugal}
\thanks{This work is supported by Fundação para a Ciência e a Tecnologia FCT/MCTES (PIDDAC) through the following Projects:  Associate Laboratory ARISE LA/P/\-0112/2020; UIDP/04044/2020; UIDB/04\-044/2020; PAMI - ROTEIRO/0328/2013 (Nº 022158); MATIS (CEN\-TRO-01-0145-FEDER-000014 - 3362); Generative.Thermodynamic; by CDRSP and ESTG from the Polytechnic Institute of Leiria.}
\email{martins.ferreira@ipleiria.pt}

\author[M. Sobral]{Manuela Sobral}
\address[Manuela Sobral]{CMUC and Departamento de
Matem\'atica, Universidade de Coimbra, 3001--501 Coimbra,
Portugal}
\thanks{The second author was partially supported by the Centre for Mathematics of
the University of Coimbra -- UID/MAT/00324/2020}
\email{sobral@mat.uc.pt}

\begin{abstract}
In the category of monoids we characterize monomorphisms that are normal, in an appropriate sense,  to internal reflexive relations,  preorders or equivalence relations. The zero-classes of such internal relations are first  described in terms of convenient syntactic relations associated to them and then through the adjunctions associated with the corresponding normalization functors. The largest categorical equivalences induced by these adjunctions   provide an equivalence between the categories of  relations generated by their zero-classes and the ones of monomorphisms that we suggest to call \emph{normal with respect to} the internal relations considered. This idea, although being transverse to the literature in the field, has not in our opinion been presented and explored in full generality. The existence of adjoints to the normalization functors permits developing a theory of normal monomorphisms, thus extending many results from groups and protomodular categories to monoids and unital categories.

%

\end{abstract}

\maketitle


\today

\section{Introduction}

In the category of groups every internal reflexive relation
is an internal equivalence relation, i.e. a congruence. In addition, the zero-classes of such congruences, that are exactly the normal subgroups, determine all the other classes and so the congruence itself.
 As it is well known, this is not the case in the category $\Mon$ of monoids where we have strict inclusions
\begin{equation}\label{eq:stric inclusions}
\Eq(\Mon) \subset\Ord( \Mon) \subset \RRel(\Mon)\end{equation}
of the full subcategories of the category $\Rel(\Mon)$ of internal relations in monoids consisting of reflexive relations, preorders and equivalence relations. Furthermore, the corresponding zero-classes
do not determine the relations, in general.

{\em Clots} \cite{AU} are the zero-classes of internal reflexive relations and {\em positive cones} is the name of the zero-classes
of preorders. The {\em normal submonoids} are the zero-classes of congruences.
In \cite{MFS}, internal transitive and reflexive relations in monoids, i.e. the preordered monoids,  denoted there by $\Ord\Mon$, were studied and an equivalence was established between the full 
subcategory $\Ord\Mon^*$ of all preordered monoids whose preorder is induced by its positive cone $P$ (where $a \leq_ P b \Leftrightarrow b \in P + a $) and the full subcategory of monomorphisms in monoids that consists of what were called there {\em the right normal mononorphisms} (those inclusions $ M \to A$ for which $ a + M \subseteq M + a$ for every $a \in A$). Here we are going to describe the larger equivalence of categories induced by the adjunction $F \dashv N$, where $N$ is the normalization functor, which, in particular, provides a categorical definition of positive cone.

The categorical definition of a clot was given in \cite{JMU1} for semi-abelian categories and then in \cite{JMU2} in pointed regular categories with finite coproducts.  In Section 3, we consider clots in what seems to be the more general context where the definition makes sense and its correspondence to internal reflexive relations, called semicongruences there, can be studied.

Inspired by the so-called {\em syntactic congruence}
and {\em syntactic preorder} for subsets $M$ of a free monoid $A$ (see \cite{SE, P}), we define
what we call the {\em M-congruence} and the {\em M-preorder}, when $M$ is a submonoid of an arbitrary monoid $A$. The name is justified by the fact that these relations are completely determined by $M$. Furthermore, they characterize the submonoids of $A$ that are the zero classes of congruences and of preorders, respectively (Propositions \ref{prop: M congruence} and  \ref{prop:M preorder}).

We also define, for each submonoid $M$ of $A$, a reflexive relation $R_M$.
Whenever $R_M$ is an internal relation, in which case  we call it the {\em $M$-reflexive relation} of the submonoid $M$,  we obtain a completely analogous characterization of the submonoids that are clots (Proposition \ref{prop: M reflexive}).
There is a large class of monoids for which these reflexive relations are indeed internal that includes the finite and the Dedekind finite monoids. However, it is an open problem to decide whether this holds or not for all monoids as stated in Example \ref{eg:10} (see Section \ref{sec:examples}).

So natural (and useful) these relations may be, they are not functorial: there is no functor with these functions of objects as we show in
Example \ref{eg:8}.
In contrast, in Section 3, we describe such functorial relations as left adjoints to the three normalization functors $$N\colon \A \to \Mono(\Mon),$$  where $\A$ is one of the categories of internal relations referred to above as instances of a more general situation (Theorem \ref{thm:1}).

The largest equivalence of categories induced by each of these adjunctions $F \dashv N (\eta, \varepsilon)$, i.e. $Fix(\varepsilon) \sim Fix(\eta)$, between the full subcategories with all the objects for which the counit or the unit are isomorphisms, respectively, have a nice interpretation.  In each case, the full subcategory $Fix(\varepsilon)$ consists of all the corresponding relations that are generated by the zero-class (in the sense that such a relation is the smallest one
with the same zero-class)
and $Fix(\eta)$ is the full subcategory $\NMono(\Mon)$ of $\Mono(\Mon)$ consisting of what we call  {\em the normal monomorphisms} with respect to objects of $\A$ that are exactly the clots, the positive cones and the normal monomorphisms, respectively.


The main purpose of this paper is to provide a systematic study of normal monomorphisms relative to specific types of internal binary (endo)relations and it is
is organized as follows. In Section 2, specific congruences, internal preorders and reflexive relations on a monoid $A$, associated with each submonoid $M$ of $A$, are introduced in order to characterize their zero-classes. In Section 3, we define left adjoints to the corresponding normalization functors as a specialization of a more general situation in the context of a pointed, wellpowered and finitely complete category $\C$ with arbitrary intersections of subobjets and categories $\tau(\C)$ of binary relations of a fixed type $\tau$ on objects of $\C$, satisfying some conditions. In pointed varieties $\V$, the smallest internal relations on an object $A$ of $\V$ whose zero class contains a subobject $M$ of $A$ can be described as special subobjects of $A \times A$, as we do here when $\V=\Mon$, providing an alternative way to describe these left adjoints, a straightforward procedure that contributes to better exhibit the differences between the corresponding zero-classes.
 We put together our observations in three theorems that, in particular, give a systematic way to construct for each monomorphism that is normal to one of the internal relations considered above the smallest internal relation to which it is normal to.

 In the last section, we present several examples illustrating our previous claims.

\section{Syntactic relations}\label{sec:2}

In this section we define binary relations from the syntactical point of view that will play a role in the characterizations of clots, positive cones and normal submonoids presented in Theorems~\ref{thm:E}--\ref{thm:R}.

The syntactic monoid is widely used in computation theory due to its applications in automata theory and regular languages (see \cite{SE,H,Syntactic, P}). According to Eilenberg (\cite{SE}, p.75), the first clear-cut and systematic exposition of the syntactic monoid can be found in \cite{Syntactic}.
The syntactic congruence and  syntactic preorder are usually defined for subsets $M$ of a free monoid $\Sigma^{*}$. For our purposes we will consider an arbitrary monoid $A$ instead of a free monoid and submonoids $M$ of $A$ rather than subsets.
However, both the syntactic congruence and the syntactic preorder are internal relations regardless of whether $M$ is a submonoid or just a subset of $A$. The relevance of this aspect in our study contrasts with the one usually needed in computer science. Instead of asking if a subset $M$ of a monoid $A$ is recognizable (in the sense of \cite{SE}, p. 68), we will be interested in knowing whether $M$ is or not the zero-class of the syntactic congruence or the syntactic preorder, in which cases $M$ is necessarily a submonoid of $A$.

\subsection{The syntactic congruence}

Given a subset $M$ of a monoid $A$, the  syntactic  {\em M-congruence} is defined by $a \sim_M b$
 if, for every $x, y \in A$,
\[ x a y \in M \Leftrightarrow x b y \in M.\]
This binary relation is an equivalence relation. Furthermore, it is compatible with the monoid operation:
if $a \sim_M b$ and $a' \sim_M b'$  then,  for every $x, y \in A$,
\[ x a (a' y)  \in M \Leftrightarrow x b (a' y)=(xb)a'y \in M \Leftrightarrow (x b) b' y \in M. \]

In other words, the set $E=\{(a,b)\in A\times A\mid a\sim_{M}b\}$ is a submonoid of $A \times A$, and being an equivalence relation, it is a congruence on $A$. The syntactic monoid is the quotient $A/E$ and the zero-class of $E$,
\[[1]_{E}=\{u\in A\mid (1,u)\in E\},\]
is a submonoid of $A$.

\begin{proposition}\label{prop: M congruence} Let $M$ be a submonoid of the monoid $A$. Then $M$ is the zero-class of some congruence in $A$ if and only if the following condition holds

(C) For every $x, y \in A$ and every $u \in M$,
\[x  y \in M \Leftrightarrow x u y \in M.\]
In this case, $M$ is exactly the zero class of $\sim_M$.
\end{proposition}

\begin{proof} The zero class of $\sim_M$ is contained in $M$. Indeed, if $1 \sim_M u$ then, for every $x, y \in A$,
\[ x y \in M \Leftrightarrow x u y \in M.\]
So, in particular, for $x=y=1$,  $ 1 \in M \Leftrightarrow u \in M$.

Furthermore, $M \subseteq [ 1 ]_{\sim_M}$ whenever $(C)$ holds: if $u \in M$ then it belongs to the zero class of the $M$-congruence when
$$x y \in M \Leftrightarrow x u y \in M,$$
that is exactly condition $(C)$.

If $M = [ 1 ]_ {\sim}$ for some congruence $\sim $ on $A$ then if $1 \sim u$, by reflexibility and compatibility,  we conclude that $ x y \sim x u y$. Since $x y \in M$ means that $1 \sim x y$ then also $ 1 \sim x u y$, by transitivity of the relation, and so $x u y \in M$. Consequently, condition $(C)$ is fulfilled.

So, we conclude that $M$ is the zero class for some congruence on $A$ if and only if it is the zero class of its  $M$-congruence.
\end{proof}

\begin{remark}
{\em In \cite{FR} Facchini and Rodaro present (Thm 9), the following necessary and sufficient condition for a submonoid  $M$  of a monoid $A$ to be a  kernel:

(F) For every $x, y \in A$,
\[x M y \cap M \neq \emptyset \Rightarrow x M y  \subseteq M,\]

that is, obviously, equivalent to condition

(C) For every $x, y \in A$ and every $u \in M$,
\[x  y \in M \Leftrightarrow x u y \in M,\]

as it is  easy to prove:

(F) $\Rightarrow$ (C)

If $x, y  \in A$ and $u \in M$ then
\[ xy \in M \Rightarrow x M y \cap M \neq \emptyset\Rightarrow x M y \subseteq M \Rightarrow x u y \in M. \]
And
\[ x u y \in  x M y \cap M \Rightarrow x M y \subseteq M \Rightarrow x 1 y = x y \in M. \]

(C) $\Rightarrow$ (F)

For $x, y \in A$, if there exists an element $u \in M$ such that $x u y \in M$ then, for every $v \in M$,
\[ x u y \in M \Leftrightarrow xy \in M \Leftrightarrow x v y \in M\]
that is $x M y \subseteq M$.}
\end{remark}

Note that kernels need not coincide with normal subobjects in non-exact categories like for example the category of torsion free abelian groups (see e.g. \cite{Gran}).

\subsection{The syntactic preorder}

Every subset $M$ of a monoid $A$ induces an internal reflexive and transitive relation $\leq_M$ on $A$ called the syntactic preorder (see \cite{P}). It is defined as $a \leq_M b$ if for every $x, y \in A$
\[ x a y \in M \Rightarrow x b y \in M.\]
This binary relation is reflexive, transitive and compatible with the monoid operation. Its positive cone consists of all  $a \in A$ such that, for every $x, y \in A$,
\[ x y \in M \Rightarrow x a y \in M,\]
which implies that it is a submonoid of $A$.

\vspace{1mm}

\begin{proposition}\label{prop:M preorder} Let $M$ be a submonoid of a monoid $A$. Then $M$ is the positive cone of some preorder in $A$ if and only if the following condition holds

(P) For every $x, y \in A$ and every $u \in M$.
\[ xy \in M \Rightarrow xuy \in M.\]

In this case, $M$ is the positive cone of $\;\leq_M$.
\end{proposition}

\begin{proof} 
The proof is analogous to the one of the previous proposition and, in the same way, we conclude that
 $M$ is the zero class of some preorder on $A$ if and only if it is the zero class of its associated $M$-preorder.
\end{proof}

Example \ref{eg:8} in Section \ref{sec:examples} shows that neither of the above syntactic constructions is functorial.


\vspace{.5cm}

Let us now investigate the case of clots, that is the zero-classes of internal reflexive relations.

\subsection{The syntactic reflexive relation}


Every submonoid $M$ of a monoid $A$ induces a reflexive relation $R_M$ in $A$ defined by
$a R_M b$ if for every $x,y\in A$, \[ x a y=1\Rightarrow x b y\in M.\]

As we show next there is a large class of monoids where $R_M$ is an internal reflexive relation, in which case we call it the  syntactic {\em $M$-reflexive relation} of the submonoid $M$. We point out that we don't know if this is the case for all monoids.


\begin{proposition}\label{prop: M reflexive} Let $M$ be a submomoid of a monoid $A$. Then $M$ is the zero-class of the reflexive relation $R_M$
if and only if the following condition holds

(R) For every $x, y \in A$ and every $u \in M$
\[ xy = 1 \Rightarrow xuy \in M.\]
If $M$ is a clot then $M= [1]_{R_M}$ and the converse holds when the relation $R_M$ is internal.
\end{proposition}
\begin{proof} Like in the previous cases, we have that $[1]_{R_M}$ is a subset of $M$ and
 that $M \subseteq [1]_{R_M}$ if and only if condition (R) is satisfied. Consequently, in presence of condition (R), we have that
 $M= [1]_{R_M}$.

Let $M$ be the zero class of some internal reflexive relation $S$  on $A$.  If $1 S u$, by the reflexibility and compatibility of the relation, $ x y S x u y$. Then, if  $x y = 1 $ we have that  $1 S x u y $, that is $x u y \in M$, (R) holds and so $M \subseteq [1]_{R_M}$.

The converse is true if $R_M$ is compatible.
\end{proof}

The relation $R_M$  is compatible with the monoid operation as soon as the following extra condition is
satisfied
\begin{equation}\label{eq: (*) extra condition normal reflexive}
\forall x,y,s,t\in A, (xy=1, xs\in M, ty\in M) \Rightarrow ts\in M
\end{equation}
This follows from the condition saying that every left (or right) invertible element in $A$ is invertible:
\begin{equation}
\forall x,y\in A, (xy=1 \Rightarrow yx=1)
\end{equation}
that characterizes the Dedekind finite monoids, also called Von Neumann finite or directly finite monoids (See \cite{F}).
 Dedekind finite monoids includes groups, commutative monoids, finite monoids, cancellative monoids and others but not all monoids. For example the monoid $\Set(X, X)$ for any infinite set $X$ with respect to composition of functions as well as the bicyclic monoid are not Dedekind finite.

In Dedekind finite monoids
from $xs,ty\in M$ we get $tyxs\in M$ and, if $yx=1$ as soon as $xy=1$, then we conclude $ts\in M$.

The condition of being Dedekind finite is stronger than Condition ($\ref{eq: (*) extra condition normal reflexive}$) which is not a necessary  one for compatibility of $R_M$ as Example \ref{eg:9} shows (see Section \ref{sec:examples}).

\begin{remark}

{\em In the category of groups, condition (R) means that M is a normal subgroup of the group $A$
and the congruence $R_M$ coincides with the classical one induced by the normal subgroup M, that is

$$a\sim b  \Leftrightarrow a b^{-1} \in M.$$

We recall that $\sim$ is an equivalence relation on the underlying set of $A$ if and only if $M$ is a subgroup and that it is a congruence if and only if $M$
is a normal subgroup of $A$.

}
\end{remark}

\section{Left adjoints to the normalization functors}\label{sec:3}

In this section we are going to define functors $F \colon \Mono(\Mon) \to \A$ which are left adjoints to the corresponding normalization functors, as examples of a construction in a more general context.

Let $\C$ be a pointed, well-powered and finitely complete category with arbitrary intersections, that is with generalized pullbacks of monomorphisms.

By $\A =\tau(\C)$ we denote the category of internal binary relations of a fixed type $\tau$ on $\C$-objects with morphisms the morphisms in $\C$ compatible with the relations. In addition, we assume that such relations are stable under pullback and closed under arbitrary intersections for every object of $\C$, that is, we assume that the functor  $(\; \;)_0 \colon \tau(\C) \to \C$ that assigns to each relation its base object is a fibration and that each fiber $\tau(\C)_A$ is a complete lattice. So, in particular, we assume that, for each object $A$ in $\C$, the indiscrete relation $\nabla_A$ belongs to $\tau(\C)$.

\begin{definition}\label{def:normal mono} A monomorphism $m \colon M \to A$ is said to be normal w.r.t. $\tau(\C)$ if it is part of a pullback in $\C$
\[
\xymatrix{M \ar[r]^{u} \ar[d]_{m} & R\ar[d]^{r}\\ A\ar[r]_(.4){\langle 0, 1 \rangle} & A \times A}
\]
for a relation $(R, r)$ of type $\tau$ on $A$.
\end{definition}

It is easy to prove that such morphisms are stable under pullback when the same holds for the relations $\tau$. 
The above definition includes clots, positive cones and normal subobjects in the classical sense in any pointed variety but makes sense in every non-varietal category satisfying the prescribed conditions,
defining there the full subcategory $\NMono(\C,\tau)$ of the category $\Mono(\C)$ consisting of what we call the {\em normal monomorphisms with respect to $\tau$}.

 We point out that the categorical notion of normal monomorphism with respect to a congruence was introduced in \cite{B}. 

For each $\tau$, the normalization functor $$N \colon \tau(\C) \to \Mono(\C)$$ is defined on objects by $N(A, R) = m$ as in the diagram of Definition \ref{def:normal mono}. Indeed, if $r_i = p_i \cdot r$
with $p_i$ the direct product projections for i = 1, 2, then $ u $ is the kernel of $r_1$ and $m = r_2 \cdot u$, the classical definition of the normalization functor \cite{B}.
The universal property of pullbacks provides the definition of $N$ on morphisms.

\begin{theorem}\label{thm:1} Let $\C$ be a pointed, well-powered and finitely complete category with arbitrary intersections and  $\tau(\C)$ be a pullback-stable full subcategory of the one of internal relations on $\C$ closed under arbitrary intersections.
Then the normalization functor $N \colon \tau(\C) \to \Mono(\C)$ has a left adjoint $$F\colon \Mono(\C) \to \tau(\C)$$ defined on objects by $F( m \colon M \to A) = (R, r)$ with  $(R, r) = \cap{\S}_m$, where ${\S}_m$ is the set of all $\tau$-relations on $A$  through which $\langle 0, 1\rangle m$ factors.
\end{theorem}
\begin{proof}
 We define $F(m) = (R, r)= \cap{\S}_m$, where $\S_m$ is not empty because it contains $\nabla_A$ since $\tau(\C)$ is closed under arbitrary intersections.

The fact that the relations $\tau$ on $\C$ are stable under pullback, enables us to assign to each $(f, \overline{f})\colon  m \to m'$ a morphism
$$(f \times f, g) \colon F(m)=(R, r) \to F(m')= (R', r')$$ where $g \colon  R \to R'$ is the morphism $t \cdot l$, $(t, \overline{r'})$ being the pullback of $ r'$ along $f \times f$, which gives the relation $({\overline{R'}}, \overline{r'})$ that belongs to ${\S}_m$ and so there exists a morphism $l$ such that
$\overline{r'} \cdot l = r$:

 \begin{equation}\label{diag:A}
    \xymatrix{R\ar[r]^{l}\ar[d]_{r}&\overline{R'}\ar[d]^{\overline{r'}}\ar[r]^{t}& R'\ar[d]^{r'}\\
    A\times A \ar@{=}[r] & A\times A \ar[r]^{f\times f} & A'\times A'}
\end{equation}

This defines a functor $F$ which, in addition, is left adjoint of $N$ with unit $\eta_m= (1_A, s) \colon m \to NF(m)$
\begin{equation}\label{diag:B}
    \xymatrix{M\ar[rr]^{u}\ar[dd]_{m}\ar@{-->}[rd]^{s}&& R\ar[dd]^{F(m)}\\
    & P \ar[ur]^{}\ar[ld]^{NF(m)} \\
     A \ar[rr]_{\langle 0,1 \rangle} && A\times A}
\end{equation}
where $P$ is the pullback of $F(m)$ along $\langle 0, 1\rangle$.
\end{proof}

In the case of congruences in Barr-exact categories there is another way to describe the left adjoint to the normalization functor.

\begin{corollary} Let $\C$ be a Barr-exact category and $\tau(\C)$ denote the category $\Eq(\C)$ of internal equivalence relations on $\C$. 
Then
$F(m)$ is the kernel pair of the cokernel of $m$.
\end{corollary}
\begin{proof}  If $F(m) = (R, r)$ and $(R, r_1, r_2)$ is the kernel pair of some morphism in $\C$ then it is the kernel pair of the coequalizer $q \colon A \to B$  of $r_1, r_2 \colon R \to A$ and so, up to isomorphism,  $(R, r_1, r_2)$ coincides with $(A \times_B A, \pi_1, \pi_2)$.

It remains to prove that $q$ is the cokernel of $m$, that is, it is the coequalizer of $(m, 0)$. If $f m = f 0$ then we have that
the kernel pair of $f$ belongs to $S_m$ and so we conclude that $f r_1 = f  r_2$. Consequently, there is a unique morphism $f'$ such that $f' q = f$. 
\end{proof}


Recall that an internal binary relation $(R, r_1, r_2)$ on an object $A$ of $\C$ is
\begin{itemize}
\item {\it reflexive} if there exists a morphism $e\colon A \to A$ such that $r_1 e = r_2  e = 1_A$, that is, if $1_A \leq R$.
\item {\it symmetric}  if there exists a morphism $s\colon R \to R $ such that $r_1  s = r_2$ and $r_2  s = r_1$, that is, if $R^o \leq R$.
\item {\it transitive}  if there exists a morphism $t\colon R \times_A R \to R$ such that $r_1  t = r_1  \pi_1$ and $r_2  t = r_2  \pi_2$, that is, if $RR \leq R$ whenever $\C$ is a regular category.
\end{itemize}

If, in Theorem \ref{thm:1},  $\tau$ denotes reflexive relations, preorders or equivalence relations in $\C$ (but not orders, because the existence of the order $\nabla_A$ would imply that $A$ is singular)  then $\tau(\C)$ satisfy the prescribed conditions and so we obtain adjunctions
\begin{eqnarray}
\xymatrix{\RRel(\C)\ar@<.5ex>[r]^{N}&\Mono(\C)\ar@<.5ex>[l]^{R}}\\
\xymatrix{\Ord(\C)\ar@<.5ex>[r]^{N}&\Mono(\C)\ar@<.5ex>[l]^{P}}\\
\xymatrix{\Eq(\C)\ar@<.5ex>[r]^{N}&\Mono(\C)\ar@<.5ex>[l]^{E}}
\end{eqnarray}
denoting the functor $F$ by $R, P$ and $E$, respectively. This is the case when $\C = \Mon$. 

\vspace{2mm}

Whenever $\C$ is a pointed variety we have another classical procedure to define these left adjoints that consists of assigning to each
monomorphism $m\colon{M\to A}$, that, without lost of generality, we assume to be an inclusion, the smallest zero-class in $A$ that contains $M$,
by describing
appropriate submonoids of $A \times A$. The explicit descriptions of such submonoids, that we present next, enable us to better exhibit different aspects of the zero classes of each of the internal relations 
as shown in Example \ref{eg:2} of Section \ref{sec:examples}.

\vspace{2mm}

In the rest of this section, for simplicity of exposition, we use additive notation but we do not assume that the monoids are commutative.

By $R=R(m)$ we denote the smallest internal reflexive relation on a monoid $A$ whose zero class contains $M$, that is, the submonoid of $A \times A$ generated by the set $(\{0\} \times M) \cup  \Delta_A $. This means that $x R y$ if and only if there exists a natural number $n\in \mathbb{N}$ together with sequences $a_i\in A$ and $u_i\in M$, with $i=1,\ldots,n$, such that
\begin{equation}
x=a_1 + \ldots + a_n\text{ and } y=a_1+ u_1+ \ldots +a_n + u_n.
\end{equation}

The relation $P=P(m)$ is the smallest reflexive and transitive relation (a preorder) on $A$ whose zero class contains all the elements in $M$. In this case it is the transitive closure of the above internal reflexive relation which is again internal, as it is easy to check. We conclude that $x P y$ if and only if there exist natural numbers $m,n\in \mathbb{N}$ together with double indexed sequences $a_{ij}\in A$ and $u_{ij}\in M$, with  $i=1,\ldots,m$ and $j=1,\ldots,n$, such that
\begin{eqnarray*}
x &=& \sum_{j=1}^{n} a_{1j}\\
\sum_{j=1}^{n} \left(a_{1j}+u_{1j}\right) &=& \sum_{j=1}^{n} a_{2j}\\
\end{eqnarray*}
\begin{eqnarray*}
&\cdots&
\end{eqnarray*}
\begin{eqnarray*}
\sum_{j=1}^{n} \left(a_{(i-1)j}+u_{(i-1)j}\right) &=& \sum_{j=1}^{n} a_{ij}
\end{eqnarray*}
\begin{eqnarray*}
&\cdots&
\end{eqnarray*}
\begin{eqnarray*}
\sum_{j=1}^{n} \left(a_{(m-1)j}+u_{(m-1)j}\right) &=& \sum_{j=1}^{n} a_{mj}\\
\sum_{j=1}^{n} \left(a_{mj}+u_{mj}\right) &=& y.
\end{eqnarray*}

\vspace{2mm}

The relation $E=E(m)$ is the smallest internal equivalence relation on $A$ whose zero class contains $M$.
We start by considering the submonoid of $A\times A$ generated by $(M\times M)\cup \Delta_A$ that is the smallest internal reflexive and symmetric relation whose zero-class contains $M$. Then the transitive closure of this relation gives $E$. 
Thus we have that $x E y$ if and only if there exist natural numbers $m,n\in \mathbb{N}$ together with double indexed sequences $a_{ij}\in A$ and $u_{ij},v_{ij}\in M$, with  $i=1,\ldots,m$ and $j=1,\ldots,n$, such that
\begin{eqnarray*}
x &=& \sum_{j=1}^{n} \left(a_{1j}+v_{1j}\right)\\\
\sum_{j=1}^{n} \left(a_{1j}+u_{1j}\right) &=& \sum_{j=1}^{n} \left(a_{2j}+v_{2j}\right)
\end{eqnarray*}
\begin{eqnarray*}
&\cdots&
\end{eqnarray*}
\begin{eqnarray*}
\sum_{j=1}^{n} \left(a_{(i-1)j}+u_{(i-1)j}\right) &=& \sum_{j=1}^{n} \left(a_{ij}+v_{ij}\right)
\end{eqnarray*}
\begin{eqnarray*}
&\cdots&
\end{eqnarray*}
\begin{eqnarray*}
\sum_{j=1}^{n} \left(a_{(m-1)j}+u_{(m-1)j}\right) &=& \sum_{j=1}^{n} \left(a_{mj}+v_{mj}\right)\\
\sum_{j=1}^{n} \left(a_{mj}+u_{mj}\right) &=& y.
\end{eqnarray*}

\begin{proposition} If we denote by $N$ the normalization functor from the categories of internal reflexive relations, preorders and equivalence relations in the category $\Mon$ of monoids, we have the following adjunctions
\begin{eqnarray}
\xymatrix{\RRel(\Mon)\ar@<.5ex>[r]^{N}&\Mono(\Mon)\ar@<.5ex>[l]^{R}}\label{eq:adjoint R}\\
\xymatrix{\Ord(\Mon)\ar@<.5ex>[r]^{N}&\Mono(\Mon)\ar@<.5ex>[l]^{P}}\\
\xymatrix{\Eq(\Mon)\ar@<.5ex>[r]^{N}&\Mono(\Mon)\ar@<.5ex>[l]^{E}}
\end{eqnarray}
where the functors $R$, $P$ and $E$ are defined on objects as described above.
\end{proposition}

The largest equivalence $\Fix (\varepsilon) \sim \Fix (\eta)$ induced by each of these adjunctions, consists of  $\Fix (\varepsilon)$, the category of all relations that are generated by their zero classes in the sense that they are the smallest relations with these zero-classes, and $\Fix (\eta)$, the category of all the corresponding zero classes. In the  three particular cases of our study, $\Fix (\eta)$ consist of all clots, all positive cones and all normal submonoids.

Let us denote by $\Fix(\eta_R)$, $\Fix(\eta_P)$ and $\Fix(\eta_E)$ the full subcategories of $\Mono(\Mon)$  consisting of all monomorphisms $m\colon{M\to A}$ such that the corresponding adjunction unit $\eta_m$ is an isomorphism. Similarly, let us denote by $\Fix(\varepsilon_E)$, $\Fix(\varepsilon_P)$ and $\Fix(\varepsilon_R)$ the respective subcategories of relations that are determined by their respective zero-classes.  

The following theorems sum up our observations with respect to normal monomorphism relative to the internal binary relations we are dealing with in $\Mon$.  There, for every monomorphism $m\colon M \to A$ we always identify $m(M)$ with $M$.

\renewcommand{\theenumi}{\alph{enumi}}

\begin{theorem}[Normal subobject]\label{thm:E}
A monoid monomorphism $m\colon{M\to A}$ is a normal subobject if one, and so all, of the following equivalent conditions is satisfied:
\begin{enumerate}
\item $M$ is the zero class of some congruence;
\item $M$ is the zero-class of its syntactic congruence;
\item $M$ is the zero-class of the congruence $E(m)$;
\item $M$ is the zero-class of the smallest congruence whose zero-class contains $M$;
\item $m$ belongs to $\Fix(\eta_E)$, that is, $m$ is normal w.r.t. $\Eq(\Mon)$;
\item $E(m)$ belongs to $\Fix(\varepsilon_E)$;
\item for every $x, y \in A$ and every $u \in M$.
\[ xy \in M \Leftrightarrow xuy \in M.\]
\item for every natural numbers $m,n\in\mathbb{N}$ and every double indexed sequences $a_{ij}\in A$, $u_{ij},v_{ij}\in M$, with $i=1,\ldots,m$ and $j=1,\ldots,n$, if $$0 = a_{11}+v_{11}+\ldots+a_{1n}+v_{1n}$$ and
\begin{eqnarray*}
 a_{11}+u_{11}+\ldots +a_{1n}+u_{1n} &=& a_{21}+v_{21}+\ldots+a_{2n}+v_{21}\\
 &\cdots&\\
 a_{(i-1)1}+u_{(i-1)1}+\ldots +a_{(i-1)n}+u_{(i-1)n} &=& a_{i1}+v_{i1}+\ldots+a_{in}+v_{in}\\
 &\cdots&\\
  a_{(m-1)1}+u_{(m-1)1}+\ldots +a_{(m-1)n}+u_{(m-1)n} &=& a_{m1}+v_{m1}+\ldots+a_{mn}+v_{mn}
\end{eqnarray*}
 then $$a_{m1}+u_{m1}+\ldots +a_{mn}+u_{mn}\in M.$$
\end{enumerate}
\end{theorem}

\begin{theorem}[Positive cones]\label{thm:P}
A monoid monomorphism $m\colon{M\to A}$ is a positive cone if one, and so all, of the following equivalent conditions is satisfied:
\begin{enumerate}
\item $M$ is the zero-class of some preorder;
\item $M$ is the zero-class of its syntactic $M$-preorder;
\item $M$ is the zero-class of the preorder $P(m)$;
\item $M$ is the zero-class of the smallest preorder whose zero-class contains $M$;
\item $m$ belongs to $\Fix(\eta_P)$, that is, $m$ is normal w.r.t.  $\Ord(\Mon)$;
\item $P(m)$ belongs to $\Fix(\varepsilon_P)$;
\item\label{item:8 of Eq} for every $x, y \in A$ and every $u \in M$.
\[ xy \in M \Rightarrow xuy \in M.\]
\item\label{cond h of thm 3} for every natural numbers $m,n\in\mathbb{N}$ and every two double indexed sequences $a_{ij}\in A$, $u_{ij}\in M$, with $i=1,\ldots,m$ and $j=1,\ldots,n$, if $$0 = a_{11}+\ldots+a_{1n}$$ and
\begin{eqnarray*}
 a_{11}+u_{11}+\ldots +a_{1n}+u_{1n} &=& a_{21}+\ldots+a_{2n}\\
 &\cdots&\\
 a_{(i-1)1}+u_{(i-1)1}+\ldots +a_{(i-1)n}+u_{(i-1)n} &=& a_{i1}+\ldots+a_{in}\\
 &\cdots&\\
  a_{(m-1)1}+u_{(m-1)1}+\ldots +a_{(m-1)n}+u_{(m-1)n} &=& a_{m1}+\ldots+a_{mn}
\end{eqnarray*}
 then $$a_{m1}+u_{m1}+\ldots +a_{mn}+u_{mn}\in M.$$
\end{enumerate}
\end{theorem}

\begin{theorem}[Clots]\label{thm:R}
A monoid monomorphism $m\colon{M\to A}$ is a \emph{clot} if one, and so all, of the following equivalent conditions is satisfied:
\begin{enumerate}
\item $M$ is the zero-class of some internal reflexive relation;
\item $M$ is the zero-class of the reflexive relation $R(m)$;
\item $M$ is the zero-class of the smallest internal reflexive relation whose zero-class contains $M$;
\item $m$  belongs to $\Fix(\eta_R)$, that is, $m$ is normal w.r.t. $\RRel(\Mon)$;
\item $R(m)$ belongs to $\Fix(\varepsilon_R)$;
\item\label{cond f of thm 4} For every natural number $n\in\mathbb{N}$ and for every two sequences $a_i\in A$, $u_i\in M$, with $i=1,\ldots,n$, if $$a_1+\ldots+a_n=0$$ then $$a_1+u_1+\ldots +a_n+u_n\in M.$$
    \end{enumerate}
Furthermore, when $R_M$ is an internal relation (i.e., it is the syntactic $M$-reflexive relation) the above conditions are also equivalent to the following:
\begin{enumerate}
    \item[(g)] $M$ is the zero-class of its syntactic reflexive relation;
    \item[(h)] for every $x, y \in A$ and every $u \in M$.
\[ xy=1 \Rightarrow xuy \in M.\]
\end{enumerate}

\end{theorem}

\section{Examples}\label{sec:examples}

Examples \ref{eg:1}--\ref{eg:3} illustrate the strict inclusions displayed in $(\ref{eq:stric inclusions})$. Example \ref{eg:4} describes a normal submonoid which is not a subgroup. Examples \ref{eg:5}--\ref{eg:11} are related with observations made in Section \ref{sec:2}.

\renewcommand{\theenumi}{\Alph{enumi}}

\begin{enumerate}
\item\label{eg:1} A simple example of a submonoid that is not a clot is obtained as follows. Take the monoid  with four elements \{1,2,3,4\} and multiplication defined by
\begin{equation}
\begin{tabular}{c|cccc}
$\cdot$ & 1 & 2 & 3 & 4 \\ \hline
1 & 1 & 2 & 3 & 4 \\
2 & 2 & 2 & 3 & 3 \\
3 & 3 & 2 & 3 & 2 \\
4 & 4 & 2 & 3 & 1 \\
\end{tabular}
\end{equation}
and observe that $M=\{1,2\}$ is a submonoid. However, the smallest compatible reflexive relation $R$ containing the pair $(1,2)$  is $$\{(1,1),(1,2),(1,3),(2,2),(2,3),(3,2),(3,3),(4,2),(4,3),(4,4)\}$$ whose zero-class is the submonoid $N(R)=\{1,2,3\}$.

\item\label{eg:2} An example of a clot which is not a positive cone is provided by the monomorphism $m\colon{\N\to F(a,b)}$, from the natural numbers to the free monoid on two generators $a$ and $b$, which associates $1\in\N$ to the word $ab$. Then $m(n)=(ab)^n$ and $\N$ is the clot of a reflexive relation but it is not the positive cone of any reflexive and transitive relation on $A=F(a,b)$.

We can prove that in two different ways:

\begin{enumerate}
\item[(i)] First we observe that in $R(m)$ we have that $1Rx$ if and only if $x=(ab)^n$ for some natural number $n$. This already shows that $\N$ is a clot. But it is not a positive cone because the relation is not transitive. Indeed, otherwise, $1Raabb$ should hold since $1Rab$ and $abRaabb$, but $aabb$ is not of the form $(ab)^n$ for any natural n. The fact that $abRaabb$ follows from reflexivity and compatibility of the relation: since $aRa$, $1Rab$ and $bRb$ we obtain $abRaabb$.

\item[(ii)] Another way to see that is to use the characterizations given in Theorem \ref{thm:R} and Theorem \ref{thm:P}. First we observe that condition $(\ref{cond f of thm 4})$ of Theorem \ref{thm:R} is satisfied because $a_1 +\ldots + a_n=0$ implies that all $a_i$ have to be zero. Secondly, we observe that condition $(\ref{cond h of thm 3})$ of Theorem \ref{thm:P} is not satisfied:  taking $m=2$, $n=2$, $a_{1j}=0$, $u_{11}=(ab)$, $u_{12}=0$, $a_{21}=(a)$, $a_{22}=(b)$, $u_{21}=(ab)$, $u_{22}=0$ we observe that $(ab)=(a)+(b)$ but $(a)+(ab)+(b)=aabb$ is not of the form $(ab)^n$ and hence it is not in the image of the monomorphism $m\colon{\N\to F(a,b)}$.
\end{enumerate}

\item\label{eg:3} An example of a positive cone which is not a normal submonoid (with respect to $\Eq(\Mon)$) is provided by the inclusion of the natural numbers in the group of the integers.

\item\label{eg:4} The inclusion of the even numbers in the monoid of the natural numbers is an example of a normal submonoid. Indeed it satisfies the condition (\ref{item:8 of Eq}) of Theorem \ref{thm:E}. 

\item\label{eg:5} For every finite monoid, the relation $R_M$  is an internal reflexive relation (Section \ref{sec:2}) because these monoids are Dedekind finite.

\item\label{eg:6} The monoid $\Set (X, X)$ of endofunctions of a set $X$ for composition, if $X$ is an infinite set, is not Dedekind finite. Indeed any surjective function that is not injective has a right inverse but not a left inverse.

\item\label{eg:7} The bicylic monoid is the monoid freely generated by two elements $b$ and $c$, satisfying the equation $bc=1$. It is not Dedekind finite since $bc=1$ but $cb \neq 1$.

 \item\label{eg:8} There is no functor $\Mono(\Mon) \rightarrow \Eq (\Mon)$ assigning to each submonoid $M$ of $A$ its syntactic $M$-congruence (Section \ref{sec:2}). To show that let us consider the composition of morphisms $S_2 \rightarrow S_3 \rightarrow S_2$ of symmetric groups that gives the identity $id_{S_2}$. It induces also the identity in $\Mono(\Mon)$ of
     the monomorphism $m \colon S_2 \rightarrow S_2$, say $g f = id_m$. There is no way to define a function sending this composition to the identity of the equivalence $S_2 \times S_2$ on $S_2$. Indeed it is enough to notice that the zero class of the domain is $S_2$ and the one of $\sim_{S_2}$ in $S_3$ is just the identity $\epsilon$:
     \[ (1 3) \epsilon (1 3) \in S_2 \mbox{  but} (1 3) (1 2) (1 3) = (2 3) \notin S_2 .\] So every congruence class is a singleton.

     Consequently there is no functor sending a submonoid $M$ of $A$ to its $M$-preorder or to its $M$-reflexive relation.

\item\label{eg:9} For every submonoid $M$ of $A$ we have that if $A$ is Dedekind finite then condition $(\ref{eq: (*) extra condition normal reflexive})$ holds and that if $A$ satisfy $(\ref{eq: (*) extra condition normal reflexive})$  then $R_M$ is compatible but these two implications are strict. Indeed,
    \begin{enumerate}
    \item[(i)] If $M=A$ then condition $(\ref{eq: (*) extra condition normal reflexive})$ holds and $A$ may be one of the examples of a non Dedekind finite monoid.
    \item[(ii)] There are  internal reflexive relations $R_M$  that do not satisfy condition $(\ref{eq: (*) extra condition normal reflexive})$. Let $A= \Set(\N, \N)$ for composition of functions and $M=\{ 1_{\N}\}$. Then $R_M$ is compatible since now $a R_M b$ if, for every functions $x, y \colon \N \to \N$, $$x a y = 1  \Rightarrow x b y = 1,$$ and we can find examples of functions $x, y, s, t$  such that $xy= 1, xs =1$ and $ty=1$ but $ts$ is not equal to $ 1$. For example,

        $x \colon \N \to \N$ taking zero to zero and $n > 0$ to $n-1$;

        $y \colon \N \to \N$ sending each $n > 0$ to $n+1$;

        $s \colon \N \to \N$ taking zero to zero and $n > 0 $ to $n + 1$;

        $t \colon \N \to \N$ taking 0 to 5 and $n > 0$ to $n-1$.

    Then $xy= 1, xs =1$ and $ty=1$ but $ts(0)= 5$.

    \end{enumerate}

\item\label{eg:10} We do not know whether  $R_M$ is or is not always an internal relation. The following example brings some light on the problem but does not solve it.

Let $A$ be the monoid $\Set (\N, \N)$, of all endofunctions of the natural numbers for composition, and $M$ be the submonoid generated by the function $u = 2 \times -$. Then $M$ consists of all functions $ u^n = 2^n \times -$ for some $n \in \N_0$.

Then $u^n$ for $n > 0$ does not belong to the zero set of the relation $R_M$.
    Indeed, for $f, g \in A$ defined by $g(x) = x + 1$ and $f(x) = x - 1$ if $x > 1$ and $f(1)=1$ we have that $f \cdot g = 1_{\N}$ but $f \cdot u^n \cdot g \not\in M$ because
    $f \cdot u^n \cdot g (x) = 2^n(x+1) - 1$, for $x \in \N$.

    Thus, $[1]_{R_M} = \{ 1_{\N}\}$ and so $M$ is not the zero class of any internal reflexive relation on $A$.

\item\label{eg:11}    For any submonoid $M$ of a monoid $A$ we have the following inclusions
    $$ [1]_{R_M} \subseteq M \subseteq [1]_{R(m)}, $$
    that, in the previous example, are both strict. Indeed, not only $[1]_{R_M}$ is a proper submonoid of $M$ but also
    $$ f \cdot u^n \cdot g \in [1]_{R(m)}$$ because $1 = f \cdot g$ and $ h = f \cdot u^n \cdot g $ and so
    $1 R(m) h$ but $h$ does not belong to $M$.
\end{enumerate}

\section{Final remarks}

In \cite{JMU1}, the authors give a categorical definition of clot in semi-abelian categories that they extend in \cite{JMU2}  to the context of pointed regular categories with finite coproducts and prove that clots are the same as zero classes of internal reflexive relations, called there semicongruences.
     In this work we consider clots in the context of pointed, wellpowered and finitely complete categories with arbitrary intersections of subobjects that seems to be the more general context where studying clots versus internal reflexive relations makes sense. In particular, if  $\C = \Mon$, the semicongruence associated to each clot $m: M \to A$ in Theorem 2.2 of \cite{JMU2} is the internal relation $R(m)$ defined above.

In \cite{M}, for regular categories $\C$ with pushouts of split monomorphisms along arbitrary maps, G. Metere proves that the restriction of the bifibration $ (\; \; )_0 \colon{ \RGraph(\C) \to \C}$ to $\RRel(\C)$
is a bifibration and this, in particular, enables him to construct a left adjoint to the normalization functor. When $\C$ is the category of monoids, this gives another way to construct the left adjoint (\ref{eq:adjoint R}) and, in \cite{BM}, for
Mal'tsev categories, this provides the left adjoint to $N \colon \Eq(\C) \to \Mono(\C)$.

In order to cover ideals, as a generalization of clots  \cite{AU,NMF.AM.AU.VdL}, we would need to replace the study of binary internal relations by the study of jontly monic spans of which one leg is surjective, which would require a different approach to the one followed here.

In the present work we have shown that the notion of normal mono\-morphism can be defined as the zero-class of certain types of internal binary relations. This means that every subcategory $\mathcal{R}$ of internal binary relations on a pointed category $\C$, for which the normalization functor admits a left adjoint, gives rise to a class of normal monomorphisms as the zero classes of relations in $\mathcal{R}$. We have studied the particular case where $\C=\Mon$ and $\mathcal{R}$ consists of reflexive relations, preorders, and congruences, from which the respective zero-classes are clots, positive cones, and classical normal subobjects. It is clear that other situations may be of interest too, namely reflexive and symmetric relations.




\begin{thebibliography}{99}

\bibitem{AU} P. Agliano and A. Ursini, \emph{Ideals and other generalizations of congruence classes}, J. Austral. Math. Soc. 53, (ser. A), (1992), 103--115.


\bibitem{B} D. Bourn, {\em Normal subobjects and abelian objects in protomodular categories}, J. of Algebra, 228 (2000), 143--164.



\bibitem{BM} D. Bourn and G. Metere, {\em A note on the categorical notion of normal subobject and of equivalence class}, Theory and Appl. of Cateories, Vol.
36, No. 3 (2021). 65--104.




\bibitem{Gran} M. Gran \emph{Notes on regular, exact and additive categories}, Sep 2014.


\bibitem{SE} S. Eilenberg, {\em Automata, Languages and Machines}, Vol. A, Columbia University, New York.

 \bibitem{FR} A. Facchini and E. Rodaro, {\em Equalizers and kernels in categories of monoids}, Semigroup Forum  95 (2017), 455--474.

  \bibitem{F} C. Faith, {\em Dedekind finite rings and a theorem of Kaplansky}, Communications in Algebra 31-9 (2003), 4175--4178

\bibitem{H}  W.M.L. Holcombe, {\em Algebraic automata theory}, Cambridge Studies in Advanced Mathematics. 1. Cambridge University Press (1982).

\bibitem{JMU1} G. Janelidze, L. M\'{a}rki, A. Ursini,  {\em Ideals and clots in universal algebra and in semi-abelian categories}, J. Algebra, 307 (2007), 191--208.

\bibitem{JMU2} G. Janelidze, L. M\'{a}rki, A. Ursini,  
 {\em Ideals and clots in pointed regular categories}, Appl. Categor. Struct. 17 (2009), 345 -- 350.


\bibitem{NMF.AM.AU.VdL} N. Martins-Ferreira,  A. Montoli, A. Ursini and T. Van der Linden, \emph{What is an ideal a zero-class of?}, J. Algebra Appl. {16} (3) (2017), 1--16.


\bibitem{MFS} N. Martins-Ferreira and M. Sobral, {\em Schreier split extensions of preordered monoids}, Journal of Logical and Algebraic Methods in Programming {120} (DOI:10.1016/j.jlamp.2021.100643) (2021) 

\bibitem{Syntactic} R. McNaughton and S. Papert, \emph{The syntactic monoid of a regular event}, in "Algebraic Theory of Machines, Languages and Semigroups" (M. A. Arbib, ed.), pp. 297--312, Academic Press, New York, 1968.


\bibitem{M} G. Metere, {\em Bourn-normal monomorphisms in regular Mal'tsev categories}, Theory and Apl. of Categories, 32 (2017), 122--147.

\bibitem{P} J-E. Pin, {\em Syntactic semigroups}.  In Rozenberg, G.; Salomaa, A. (eds.). Handbook of Formal Language Theory . Springer-Verlag. pp. 679–-746.

\end{thebibliography}

\end{document}